\documentclass[11pt,twoside]{amsart}

\usepackage{latexsym}
\usepackage{amsmath}
\usepackage{amsfonts}
\usepackage{amssymb}
\usepackage{hyperref}
\usepackage{mathrsfs}

\usepackage{float}
\usepackage{subcaption}

\usepackage{tikz}
\tikzset{
	element/.style={circle,fill=black,scale=0.5}
}
\usetikzlibrary{calc}
\usepackage[mathscr]{euscript}

\usepackage{enumerate}

\usepackage{color}

\usepackage{xcolor}

\newtheorem{theorem}{Theorem}[section]
\newtheorem{lemma}[theorem]{Lemma}

\newtheorem{proposition}[theorem]{Proposition}
\newtheorem{sublemma}{}[theorem]

\numberwithin{equation}{section} 

\newcommand{\cl}{{\rm cl}}
\newcommand{\si}{{\rm si}}
\newcommand{\co}{{\rm co}}

\newcommand{\delete}{\backslash}

\usepackage[T1]{fontenc}
\usepackage[utf8]{inputenc}

\makeatletter
\g@addto@macro\bfseries{\boldmath} 
\makeatother

\title[A splitter theorem for elastic elements]{A splitter theorem for Elastic Elements in $3$-Connected Matroids}

\author[G.\ Drummond]{George Drummond}
\address{School of Mathematics and Statistics, University of Canterbury, Christchurch, New Zealand}
\email{george.drummond@pg.canterbury.ac.nz}

\author[C.\ Semple]{Charles Semple}
\address{School of Mathematics and Statistics, University of Canterbury, Christchurch, New Zealand}
\email{charles.semple@canterbury.ac.nz}



\thanks{The first author was supported by a University of Canterbury PhD Publishing Scholarship, and the second author was supported by the New Zealand Marsden Fund.}

\date{\today}

\begin{document}

\thispagestyle{empty}

\begin{abstract}
An element $e$ of a $3$-connected matroid $M$ is \emph{elastic} if $\si(M/e)$, the simplification of $M/e$, and $\co(M\delete e)$, the cosimplification of $M\delete e$,  are both $3$-connected. It was recently shown that if $|E(M)|\geq 4$, then $M$ has at least four elastic elements provided $M$ has no $4$-element fans and no member of a specific family of $3$-separators. In this paper, we extend this wheels-and-whirls type result to a splitter theorem, where the removal of elements is with respect to elasticity and keeping a specified $3$-connected minor. We also prove that if $M$ has exactly four elastic elements, then it has path-width three. Lastly, we resolve a question of Whittle and Williams, and show that past analogous results, where the removal of elements is relative to a fixed basis, are consequences of this work.
\end{abstract}

\maketitle

\section{Introduction}

Tutte's Wheels-and-Whirls Theorem~\cite{tut66} and its extension, Seymour's Splitter Theorem~\cite{sey80}, are inductive tools that have been crucial in the advancement of matroid theory. Since their publication, variants of these theorems have been established whereby additional constraints are imposed on the elements that are available for removal. For example, Oxley et al.~\cite{oxl08} showed that if $B$ is a given basis of a $3$-connected matroid $M$ with no $4$-element fans, and $N$ is a $3$-connected minor of $M$, then, provided a certain necessary condition holds, there is either an element $b\in B$ such that $M/b$ is $3$-connected with an $N$-minor or an element $b^*\in E(M)-B$ such that $M\delete b^*$ is $3$-connected with an $N$-minor. More recently, Drummond et al.~\cite{dru21} proved a variant of Tutte's Wheels-and-Whirls Theorem for elastic elements. An element $e$ of a $3$-connected matroid $M$ is \emph{elastic} if $\si(M/e)$, the simplification of $M/e$, and $\co(M\delete e)$, the cosimplification of $M\delete e$, are both $3$-connected. It is shown in~\cite{dru21} that if $|E(M)|\ge 4$, and $M$ has no $4$-element fans and no member of a particular family of $3$-separators, generically referred to as $\Theta$-separators, then $M$ has at least four elastic elements. As a comparison, Bixby's Lemma~\cite{bix82} says that if $e$ is an element of $M$, then either $\si(M/e)$ or $\co(M\delete e)$ is $3$-connected. In this paper, we establish a variant of Seymour's Splitter Theorem for elastic elements. Furthermore, we consider the structure of $M$ if it has exactly four elastic elements, resolve a question of Whittle and Williams~\cite{whi13}, and show that past analogous results, where the removal of elements is relative to a fixed basis, are consequences of~\cite{dru21} and the results in this paper. We next state the main results of the paper.

Let $M$ be a $3$-connected matroid and let $N$ be a $3$-connected minor of $M$. We say that an element $e$ of $M$ is \emph{$N$-elastic} if both $\si(M/e)$ and $\co(M\delete e)$ are $3$-connected and each has an $N$-minor. Furthermore, we say that an element $e$ of $M$ is \emph{$N$-revealing} if 
either $\si(M/e)$ has an $N$-minor and is not $3$-connected, or $\co(M\delete e)$ has an $N$-minor and is not $3$-connected.

As well as fans, $\Theta$-separators provide exceptions to the presence of elastic elements. We describe the latter now. For all $n\geq 2$, let $\Theta_n$ denote the matroid whose ground set is the disjoint union $W=\{w_1, w_2, \ldots, w_n\}$ and $Z=\{z_1, z_2, \ldots, z_n\}$, and whose circuits are as follows:
\begin{enumerate}[{\rm (i)}]
\item all $3$-element subsets of $W$;

\item all sets of the form $(Z-\{z_i\})\cup \{w_i\}$, where $i\in \{1, 2, \ldots, n\}$; and

\item all sets of the form $(Z-\{z_i\})\cup\{w_j, w_k\}$, where $i$, $j$, and $k$ are distinct elements of $\{1, 2, \ldots, n\}$.
\end{enumerate}
If $n=2$, then $\Theta_2$ is isomorphic to the direct sum of $U_{1, 2}$ and $U_{1, 2}$, while if $n=3$, then $\Theta_3$ is isomorphic to $M(K_4)$. More generally, it is shown in~\cite{oxl00} that $\Theta_n$ is a matroid for all $n\ge 2$. Furthermore, for the interested reader, $\Theta_n$ is precisely the matroid that underlies the operation of segment-cosegment exchange on $n$ elements~\cite{oxl00}. If $i, j\in \{1, 2, \ldots, n\}$, it is easily checked that $\Theta_n\delete w_i\cong \Theta_n\delete w_j$. Up to isomorphism, we denote the matroid $\Theta_n\delete w_i$ by $\Theta_n^-$. If $n=3$, then $\Theta^-_3$ is isomorphic to a $5$-element fan. We call the elements in $W$ and $Z$ the {\em segment} and {\em cosegment} elements, respectively, of $\Theta_n$ and $\Theta^-_n$.

Now let $M$ be a $3$-connected matroid, and suppose that $r(M)\ge 4$ and $r^*(M)\geq 4$. Let $W$ and $Z$ be rank-$2$ and corank-$2$ subsets of $E(M)$, respectively, such that  $n=\max\{|W|,|Z|\}\geq 3$. We say $W\cup Z$ is a \emph{$\Theta$-separator} of $M$ if either $M|(W\cup Z)$ or $M^*|(W\cup Z)$ is isomorphic to one of the matroids $\Theta_n$ or $\Theta_n^-$. Furthermore, if $N$ is a $3$-connected minor of $M$, we say a $\Theta$-separator \emph{reveals} $N$ in $M$ if either
\begin{enumerate}[(i)]
\item $M|(W\cup Z)\in\{\Theta_n,\Theta_n^-\}$ and at least one element of $Z$ is $N$-revealing in $M$ or, dually,

\item $M^*|(W\cup Z)\in\{\Theta_n,\Theta_n^-\}$ and at least one element of $W$ is $N^*$-revealing in $M^*$.
\end{enumerate}

The following theorem is the above mentioned analogue of Tutte's Wheels-and-Whirls Theorem for elastic elements established in~\cite{dru21}.

\begin{theorem}
\label{elastic2}
Let $M$ be a $3$-connected matroid with no $4$-element fans and no $\Theta$-separators. If $|E(M)|\geq 4$, then $M$ has at least four elastic elements.
\end{theorem}

The main result of this paper is the next theorem.

\begin{theorem}
\label{thm: main}
Let $M$ be a $3$-connected matroid with no $4$-element fans, and let $N$ be a $3$-connected minor of $M$ such that $M$ has no $\Theta$-separators revealing $N$. If $M$ has at least one $N$-revealing element, then $M$ has at least two $N$-elastic elements.
\end{theorem}

\noindent Equivalently, provided the two initial conditions hold, that is $M$ has no $4$-element fans and no $\Theta$-separators revealing $N$, Theorem~\ref{thm: main} says that either $M$ has at least two $N$-elastic elements, or whenever $\si(M/e)$ has an $N$-minor, then $\si(M/e)$ is $3$-connected, and whenever $\co(M\delete e)$ has an $N$-minor, then $\co(M\delete e)$ is $3$-connected. The requirement of Theorem~\ref{thm: main} that $M$ has at least one $N$-revealing element is a necessary one. Consider, for example, when $M$ and $N$ have the same rank. However, the requirement of the two initial conditions in the statement of Theorem~\ref{thm: main} are not completely necessary as shown by the next theorem, a refinement of Theorem~\ref{thm: main}, which describes precisely how fans and $\Theta$-separators may prevent the presence of $N$-elastic elements.

Let $M$ be a matroid on ground set $E$. A $k$-separation $(X, E-X)$ of $M$ is \emph{vertical} if $\min\{r(X), r(E-X)\}\geq k$. Let $(X, \{e\}, Y)$ be a partition of $E(M)$. We say that $(X, \{e\}, Y)$ is a {\em vertical $3$-separation} of $M$ if $(X\cup \{e\}, Y)$ and $(X, Y\cup \{e\})$ are both vertical $3$-separations of $M$, and $e\in \cl(X)\cap \cl(Y)$. Furthermore, $Y\cup \{e\}$ is \emph{maximal} if $M$ has no vertical $3$-separation $(X', \{e'\}, Y')$ such that $Y\cup \{e\}$ is a proper subset of $Y'\cup \{e'\}$. 
Theorem~\ref{elastic2} is a consequence of~\cite[Theorem 1]{dru21} which says that if $(X, \{e\}, Y)$ is a vertical $3$-separation of a $3$-connected matroid $M$ such that $Y\cup \{e\}$ is maximal, then $X$ contains at least two elastic elements of $M$ unless $X\cup \{e\}$ is a $4$-element fan or $X$ is contained in a $\Theta$-separator. The next theorem extends this latter result to $N$-elastic elements. Its proof is given in Section~\ref{sec: splitter}.

\begin{theorem}
\label{thm: maximal}
Let $M$ be a $3$-connected matroid and let $N$ be a $3$-connected minor of $M$. Let $(X, \{e\}, Y)$ be a vertical $3$-separation of $M$ such that $M/e$ has an $N$-minor and $|X\cap E(N)|\le 1$. If $(X', \{e'\}, Y')$ is a vertical $3$-separation of $M$ such that $Y\cup \{e\}\subseteq Y'\cup \{e'\}$ and $Y'\cup \{e'\}$ is maximal, then $X'$ contains at least two $N$-elastic elements unless $X'\cup \{e'\}$ is a $4$-element fan or $X'$ is contained in a $\Theta$-separator revealing $N$.
\end{theorem}

The elasticity of elements in $\Theta$-separators is described in Section~\ref{sec: prelims}. In particular, a $\Theta$-separator revealing $N$ has at most one $N$-elastic element.
To consider what happens in the statement of Theorem~\ref{thm: maximal} when $X'\cup \{e'\}$ is a $4$-element fan, let $M$ be a $3$-connected matroid. A {\em flower} $\Phi$ is a partition $(P_1, P_2, \ldots, P_n)$ of $E(M)$ such that $|P_i|\ge 2$, and $P_i$ and $P_i\cup P_{i+1}$ are $3$-separating for all $i\in \{1, 2, \ldots, n\}$, where subscripts are interpreted modulo $n$. The parts of $\Phi$ are called {\em petals}. We say $\Phi$ is {\em swirl-like} if $n\ge 4$, and $\sqcap(P_i, P_j)=1$ for all consecutive $i$ and $j$, and $\sqcap(P_i, P_j)=0$ for all non-consecutive $i$ and $j$, where
$$\sqcap(P_i, P_j)=r(P_i)+r(P_j)-r(P_i\cup P_j)$$
is the local connectivity between $P_i$ and $P_j$. For further details regarding flowers, we refer the interested reader to~\cite{oxl04}. The proof of the next theorem is also given in Section~\ref{sec: splitter}.

\begin{theorem}
\label{thm: fans}
Let $M$ be a $3$-connected matroid such that $r(M), r^*(M)\geq 4$, and let $N$ be a $3$-connected minor of $M$. Let $(X, \{e\}, Y)$ be a vertical $3$-separation of $M$ such that $M/e$ has an $N$-minor and $|X\cap E(N)|\le 1$, and let $(X', \{e'\}, Y')$ be a vertical $3$-separation of $M$ such that $Y\cup \{e\}\subseteq Y'\cup \{e'\}$ and $Y'\cup \{e'\}$ is maximal. Suppose that $X'\cup \{e'\}$ is a $4$-element fan $F=(f_1, f_2, f_3, f_4)$, where $e'=f_1$. Then the following hold:
\begin{enumerate}[{\rm (i)}]
\item If $F$ extends to a fan of size at least six, then $F$ contains no elastic elements of $M$.

\item If $F$ extends to a fan of size five but not to a fan of size six, then $F$ contains either exactly one elastic element, namely $f_3$ and this element is $N$-elastic, or $F$ contains no elastic elements.

\item If $F$ does not extend to a fan with more elements, then either it contains exactly two elastic elements, namely $f_2$ and $f_3$, and both are $N$-elastic, or it contains no elastic elements.
\end{enumerate}
Moreover, if $F$ contains no elastic elements, then, up to duality, $M$ has a swirl-like flower $(A, \{f_1, f_2\}, \{f_3,f_4\}, B)$ as shown geometrically in Fig.~\ref{fig: flower}, or there are elements $e$ and $g$ such that $M|(F\cup \{e, g\})\cong M(K_4)$.
\label{fans1}
\end{theorem}

\begin{figure}[ht]
	\centering
	\begin{tikzpicture}[scale=1]
	\coordinate (cntrl1) at (-50:2.5) {};
	\coordinate (cntrl2) at (-20:2.5) {};
	\begin{scope}[every node/.style=element]
	\node (e1) at (-210:2) {};
	\coordinate (c2) at (90:1) {};
	\coordinate (c1) at (30:2);
	\node (e3) at (0,4) {};
	\coordinate (b) at (-120:2) {};
	
	\coordinate (cntr1) at ($(e1) + (230:1.5)$);
	\coordinate (cntr2) at ($(c1) + (-60:1.5)$);

	\draw (c1) to (e1);
	\node (e2) at ($(e1)!0.5!(e3)$) {};

	\draw (c1) to  (b);
	\draw (e1) to (b);
	
	\draw (e1) .. controls (cntr1) .. (b);
	\draw (c1) .. controls (cntr2) .. (b);
	\draw (e1) to (e2) to (e3);
	\draw (e2) to (c1);
	\draw (e3) to (c2);
	
	\node (e4) at (intersection of e3--c2 and e2--c1) {};
	
	\end{scope}
	\node[scale=1.0] at ($(e1)+(150:0.5)$) {$f_1$};
	\node[scale=1.0] at ($(e2)+(150:0.5)$) {$f_3$};
	\node[scale=1.0] at ($(e3)+(150:0.5)$) {$f_2$};
	\node[scale=1.0] at ($(e4)+(30:0.5)$) {$f_4$};
	
	\node[scale=1.5] at ($(e1)+(-100:1)$) {$A$};
	\node[scale=1.5] at ($(c1)+(-110:1)$) {$B$};
	
	\end{tikzpicture}
	\caption{The swirl-like flower $(A, \{f_1, f_2\},\{f_3,f_4\}, B)$ of Lemma \ref{elastic fans}, where, if $|F|=5$, then $f_5$ is an element in $B$.}
	\label{fig: flower}
\end{figure}
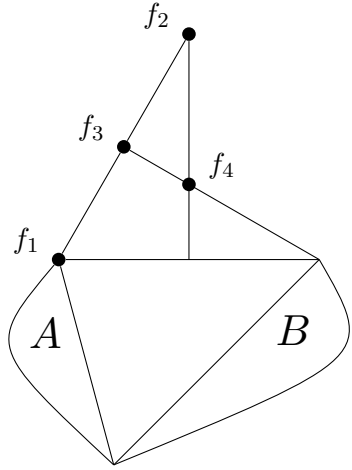

Having established a lower bound on the number of $N$-elastic elements by Theorem~\ref{thm: main}, and having a lower bound on the number of elastic elements by Theorem~\ref{elastic2}, it is natural to consider those matroids with the minimum number of such elements. 
Let $M$ be a matroid. An exact $3$-separating partition $(X, Y)$ of $E(M)$ is a \emph{sequential $3$-separation} if there is an ordering $(e_1, e_2, \ldots, e_k)$ of $X$ or $Y$ such that $\{e_1, e_2, \ldots, e_i\}$ is $3$-separating for all $i\in\{1, 2, \ldots, k\}$. A matroid has \emph{path-width three} if its ground set is sequential; that is, there is an ordering $(e_1,e_2,\ldots,e_n)$ of its ground set such that $\{e_1,e_2,\ldots,e_i\}$ is $3$-separating for all $i\in \{1,2,\ldots,n\}$. Furthermore, a \emph{path of $3$-separations} in a $M$ is an ordered partition $(P_0, P_1, \ldots, P_k)$ of $E(M)$ with the property that $P_0\cup P_1\cup \cdots \cup P_i$ is exactly $3$-separating for all $i\in \{0, 1, \ldots, k-1\}$. The proofs of the next two theorems are given in Section~\ref{sec: min elts}.

\begin{theorem}
\label{min eltsI}
Let $M$ be a $3$-connected matroid with no $4$-element fans and no $\Theta$-separators. If $M$ has exactly four elastic elements, then $M$ has path-width three.
\end{theorem}

\begin{theorem}
\label{min eltsII}
Let $M$ be a $3$-connected matroid with no $4$-element fans, and let $N$ be a $3$-connected minor of $M$ with $|E(N)|\geq 4$ such that $M$ has no $\Theta$-separators revealing $N$. Let $K$ be the set of $N$-revealing elements of $M$. If $M$ has exactly two $N$-elastic elements $s_1$ and $s_2$, then $K$ has an ordering $(e_1, e_2, \ldots, e_k)$ such that
$$(\{s_1, s_2\}, \{e_1\}, \{e_2\}, \ldots, \{e_k\}, E(M)-K\cup\{s_1, s_2\})$$
is a path of $3$-separations in $M$. Moreover, for every such ordering of $K$, both $M/e_i$ and $M\delete e_i$ have an $N$-minor for all $i<k$.
\end{theorem}

The results of this paper and~\cite{dru21} have links to the study of matroids in which the removal of elements is relative to a fixed basis~\cite{bre14, oxl08, whi13}. Let $M$ be a $3$-connected matroid, and suppose that we are given a matrix representation of $M$ in standard form relative to some basis $B$. As well as keeping $3$-connectivity, it is often desirable to remove elements from $M$ while keeping the information displayed by this representation. In particular, we want to avoid pivoting. To this end, we can remove elements either by contracting elements in $B$ or deleting elements in $E(M)-B$. Extending the results in~\cite{oxl08}, Whittle and Williams~\cite{whi13} showed that if $|E(M)|\geq 4$ and $M$ has no $4$-element fans, then $M$ has at least four elements $e$ such that either $e\in B$ and $\si(M/e)$ is $3$-connected, or $e\in E(M)-B$ and $\co(M\delete e)$ is $3$-connected. Furthermore, Brettell and Semple~\cite{bre14} gave a splitter theorem analogue of this result. In Section~\ref{sec: fixed basis}, we show that these results are implied by the work of the current paper and its predecessor~\cite{dru21}.

The paper is organised as follows. The next section contains some preliminaries on connectivity and elastic elements. Section~\ref{sec: splitter} consists of the proofs of Theorems~\ref{thm: main}--\ref{fans1}. Section~\ref{sec: min elts} considers those matroids with the minimum number of elastic elements as well as those matroids with the minimum number of $N$-elastic elements, and establishes Theorems~\ref{min eltsI} and~\ref{min eltsII}. Lastly, Section~\ref{sec: fixed basis} contains new proofs of the main results of \cite{bre14} and~\cite{whi13}, and resolves a question posed in~\cite{whi13}. Throughout the paper, the notation and terminology will follow~\cite{ox11}. In addition, if a matroid $M$ has a $\Theta$-separator, then we will implicitly assume that $M$ has rank and corank at least four.

\section{Preliminaries}
\label{sec: prelims}

\subsection*{Connectivity}

Let $M$ be a matroid with ground set $E$. The \emph{connectivity function} $\lambda_M$ of $M$ is defined on all subsets $X$ of $E$ by
$$\lambda_M(X)=r(X)+r(E-X)-r(M).$$
A subset $X$ of $E$ or a partition $(X, E-X)$ is \emph{$k$-separating} if $\lambda_M(X)\le k-1$ and is {\em exactly $k$-separating} if $\lambda_M(X)=k-1$. A $k$-separating partition $(X, E-X)$ is a \emph{$k$-separation} if $\min \{|X|, |E-X|\}\geq k$. A matroid is \emph{$n$-connected} if it has no $k$-separations for all $k < n$. Furthermore, two $k$-separations $(X_1, Y_1)$ and $(X_2, Y_2)$ of $M$ are said to \emph{cross} if each of the intersections $X_1\cap Y_1$, $X_1\cap Y_2$, $X_2\cap Y_1$, and $X_2\cap Y_2$ is non-empty. The next lemma is a particularly useful tool for handling crossing separations and follows from the fact that the connectivity function of a matroid is submodular. We refer to an application of this lemma as \emph{by uncrossing}.

\begin{lemma}
\label{uncrossing}
Let $M$ be a $k$-connected matroid, and let $X$ and $Y$ be $k$-separating subsets of $E(M)$.
\begin{enumerate}[{\rm (i)}]
\item If $|X\cap Y|\geq k-1$, then $X\cup Y$ is $k$-separating.

\item If $|E(M)-(X\cup Y)|\geq k-1$, then $X\cap Y$ is $k$-separating.
\end{enumerate}
\end{lemma}

The following three lemmas are used throughout the paper. The first follows from orthogonality. The second is a consequence of the first, while the third is a consequence of the first and second.

\begin{lemma}
\label{orthogonality}
Let $(X, \{e\}, Y)$ be a partition of the ground set of a matroid $M$. Then $e\in \cl(X)$ if and only if $e\not\in \cl^*(Y)$.
\end{lemma}

\begin{lemma}
\label{3sep1}
Let $X$ be an exactly $3$-separating set in a $3$-connected matroid $M$, and suppose that $e\in E(M)-X$. Then $X\cup \{e\}$ is $3$-separating if and only if $e\in \cl(X)\cup \cl^*(X)$.
\end{lemma}

\begin{lemma}
Let $(X, Y)$ be an exactly $3$-separating partition of a $3$-connected matroid $M$, and suppose that $|X|\ge 3$ and $e\in X$. Then $(X-\{e\}, Y\cup \{e\})$ is exactly $3$-separating if and only if $e$ is in exactly one of $\cl(X-\{e\})\cap \cl(Y)$ and $\cl^*(X-\{e\})\cap \cl^*(Y)$.
\label{3sep2}
\end{lemma}

\subsection*{Vertical and cyclic connectivity}

Recalling the terminology and notation from the introduction, a $k$-separation $(X, Y)$ of a matroid $M$ is \emph{vertical} if $\min\{r(X), r(Y)\}\geq k$. A partition $(X, \{e\}, Y)$ of $E(M)$ is a \emph{vertical $3$-separation} of $M$ if $(X\cup \{e\}, Y)$ and $(X, Y\cup \{e\})$ are both vertical $3$-separations of $M$ and $e\in \cl(X)\cap \cl(Y)$. Furthermore, $Y\cup \{e\}$ is {\em maximal} in this separation if there is no vertical $3$-separation $(X', \{e'\}, Y')$ of $M$ such that $Y\cup \{e\}$ is a proper subset of $Y'\cup \{e'\}$. A $k$-separation $(X, Y)$ of $M$ is {\em cyclic} if both $X$ and $Y$ contain circuits. Vertical $k$-separations and cyclic $k$-separations are dual concepts. In particular, it is straightforward to show (for example, see \cite[Lemma 9]{dru21}) that if $(X, Y)$ is a partition of the ground set of a $k$-connected matroid $M$, then $(X,Y)$ is a cyclic $k$-separation of $M$ if and only if $(X, Y)$ is a vertical $k$-separation of $M^*$. As such, we say a partition $(X, \{e\}, Y)$ of the ground set of a $3$-connected matroid $M$ is a \emph{cyclic $3$-separation} if $(X, \{e\}, Y)$ is a vertical $3$-separation of $M^*$.

The first of the next two lemmas indicates why vertical $3$-separations and cyclic $3$-separations arise in the context of non-elastic elements. In combination with the duality link between vertical $3$-separations and cyclic $3$-separations, the first lemma is a straightforward strengthening of~\cite[Lemma~3.1]{oxl08} and the second lemma follows from Lemma~\ref{3sep2}.

\begin{lemma}
\label{vertical1}
Let $M$ be a $3$-connected matroid, and suppose that $e\in E(M)$. Then $\si(M/e)$ is not $3$-connected if and only if $M$ has a vertical $3$-separation $(X, \{e\}, Y)$. Dually, $\co(M\delete e)$ is not $3$-connected if and only if $M$ has a cyclic $3$-separation $(X, \{e\}, Y)$.
\end{lemma}

\begin{lemma}
\label{vertical2}
Let $M$ be a $3$-connected matroid. If $(X, \{e\}, Y)$ is a vertical $3$-separation of $M$, then $(X -\cl(Y), \{e\}, \cl(Y)-\{e\})$ is also a vertical $3$-separation of $M$. Dually, if $(X, \{e\}, Y)$ is a cyclic $3$-separation of $M$, then $(X-\cl^*(Y), \{e\}, \cl^*(Y)-\{e\})$ is also a cyclic $3$-separation of $M$.
\end{lemma}

\subsection*{Fans, segments, and \texorpdfstring{$\Theta$}{Theta}-separators}

Let $M$ be a $3$-connected matroid. A subset $F$ of $E(M)$ with at least three elements is a \emph{fan} if there is an ordering $(f_1, f_2, \ldots, f_k)$ of the elements of $F$ such that
\begin{enumerate}[(i)]
\item for all $i\in \{1, 2, \ldots, k-2\}$, the triple $\{f_i, f_{i+1}, f_{i+2}\}$ is either a triangle or a triad, and

\item for all $i\in \{1, 2, \ldots, k-3\}$, if $\{f_i, f_{i+1}, f_{i+2}\}$ is a triangle, then $\{f_{i+1}, f_{i+2}, f_{i+3}\}$ is a triad, while if $\{f_i, f_{i+1}, f_{i+2}\}$ is a triad, then $\{f_{i+1}, f_{i+2}, f_{i+3}\}$ is a triangle.
\end{enumerate}
If $k\geq 4$, then the elements $f_1$ and $f_k$ are the \emph{ends} of $F$. Furthermore, if $\{f_1, f_2, f_3\}$ is a triangle, then $f_1$ is a {\em spoke-end}; otherwise, $f_1$ is a {\em rim-end}. It is elementary to show that if $F=(f_1, f_2, f_3, f_4)$ is a $4$-element fan with spoke-end $f_1$ in a $3$-connected matroid $M$ of rank at least four, then $(\{f_2, f_2, f_4\}, \{f_1\}, E(M)-F)$ is a vertical $3$-separation of $M$ in which $E(M)-\{f_2, f_3, f_4\}$ is maximal. The next result is \cite[Lemma~13]{dru21} and details when elements of a $4$-element fan are elastic.
\begin{lemma}
\label{elastic fans}
Let $M$ be a $3$-connected matroid such that $r(M),r^*(M)\geq 4$, and let $F=(f_1, f_2, \ldots, f_n)$ be a maximal fan of $M$.
\begin{enumerate}[{\rm (i)}]
\item If $n\ge 6$, then $F$ contains no elastic elements of $M$.

\item If $n=5$, then $F$ contains either exactly one elastic element, namely $f_3$, or no elastic elements of $M$.

\item If $n=4$, then $F$ contains either exactly two elastic elements, namely $f_2$ and $f_3$, or no elastic elements of $M$.
\end{enumerate}
Moreover, if $n\in \{4, 5\}$ and $F$ contains no elastic elements, then, up to duality, $M$ has a swirl-like flower $(A, \{f_1, f_2\}, \{f_3,f_4\}, B)$ as shown geometrically in Fig.~\ref{fig: flower}, or $n=5$ and there is an element $g$ such that $M|(F\cup\{g\})\cong M(K_4)$.
\end{lemma}

A \emph{segment} of a matroid $M$ is a subset $X$ of $E(M)$ such that $M|X$ is isomorphic to a rank-$2$ uniform matroid. The next lemma is elementary and is used repeatedly in this paper.

\begin{lemma}
\label{segdelete}
Let $M$ be a $3$-connected matroid and let $L$ be a segment of $M$ with at least four elements. If $\ell\in L$, then $M\delete \ell$ is $3$-connected.
\end{lemma}

Recall the definition of a $\Theta$-separator given in the introduction. The next four lemmas are extracted from~\cite{dru21}. The first lemma follows from the proofs of~\cite[Lemma~13]{dru21} and~\cite[Lemma~16]{dru21}, while the second lemma is an immediate consequence of the first. The third lemma is~\cite[Theorem~1]{dru21}. The fourth lemma follows by combining~\cite[Lemma~18]{dru21} and~\cite[Lemma~19]{dru21}.

\begin{lemma}
Let $M$ be a $3$-connected matroid, and let $S$ be a $\Theta$-separator of $M$ such that $M|S$ is isomorphic to either $\Theta_n$ or $\Theta^-_n$, where $n\ge 3$. Let $W$ and $Z$ be the set of segment and cosegment elements of $M|S$, respectively. If $w\in W$, then $\si(M/w)$ is not $3$-connected. Furthermore, if  $z\in Z$, then $\co(M\delete z)$ is not $3$-connected, unless there is no element $x\in \cl(W)$ such that $(Z-\{z\})\cup \{x\}$ is a circuit.
\label{exception}
\end{lemma}

\begin{lemma}
\label{theta elastic}
Let $M$ be a $3$-connected matroid and let $S$ be a $\Theta$-separator of $M$. If $M|S\cong \Theta_n$ for some $n\geq 3$, then no element of $S$ is elastic. If $M|S\cong \Theta_n^-$ for some $n\geq 3$, then $S$ contains a unique elastic element unless there exists an element $e\in E(M)-S$ such that $M|(S\cup\{e\})\cong \Theta_n$.
\end{lemma}

\begin{lemma}
\label{elastic1} 
Let $M$ be a $3$-connected matroid with a vertical $3$-separation $(X, \{e\}, Y)$ such that $Y\cup \{e\}$ is maximal. If $X\cup \{e\}$ is not a $4$-element fan and $X$ is not contained in a $\Theta$-separator, then at least two elements of $X$ are elastic.
\end{lemma}

\begin{lemma}
Let $M$ be a $3$-connected matroid with a vertical $3$-separation $(X, \{e\}, Y)$ such that $Y\cup \{e\}$ is maximal. If $X$ is contained in a $\Theta$-separator $S$, then either
\begin{enumerate}[{\rm (i)}]
\item $X$ is a rank-$3$ cocircuit, $M^*|S$ is isomorphic to either $\Theta_n$ or $\Theta^-_n$, where $n=|X\cup \{e\}|-1$, and there is a unique element $x\in X$ such that $x$ is a segment element of $M^*|S$ and $(X-\{x\})\cup \{e\}$ is the set of cosegment elements of $M^*|S$, or

\item $X\cup \{e\}$ is a circuit, $M|S$ is isomorphic to either $\Theta_n$ or $\Theta^-_n$ for some $n\in \{|X|, |X|+1\}$, and $X$ is a subset of the cosegment elements of $M|S$.
\end{enumerate}
\label{lem: thetamax}
\end{lemma}

\subsection*{Separations and minors}

The following two lemmas concern $3$-connected minors across $2$-separations. The first is elementary, and the second is a slight strengthening of \cite[Lemma 4.5]{bre14} and follows from the proof of that lemma.

\begin{lemma}
\label{2sep}
Let $(X, Y)$ be a $2$-separation of a connected matroid $M$, and let $N$ be a $3$-connected minor of $M$. Then $\{X, Y\}$ has a member $U$ such that $|U\cap E(N)|\leq 1$. Moreover, if $u\in U$, then
\begin{enumerate}[{\rm (i)}]
\item $M/u$ has an $N$-minor if $M/u$ is connected, and
\item $M\delete u$ has an $N$-minor if $M\delete u$ is connected.
\end{enumerate}
\end{lemma}

\begin{lemma}
Let $M$ be a $3$-connected matroid, and let $N$ be a $3$-connected minor of $M$. Let $(X,\{e\},Y)$ be a vertical $3$-separation of $M$ such that $M/e$ has an $N$-minor, where $|X\cap E(N)|\leq 1$ and $Y\cup \{e\}$ is closed. Then $M/x$ has an $N$-minor for every element $x$ of $X$, and there is at most 
one element of $X$, say $x'$, such that $M\delete x'$ has no $N$-minor. Moreover, if such an element $x'$ exists, then $x'\in \cl^*(Y)$ and $e\in \cl(X-\{x'\})$.
\label{BS lem 4.5+}
\end{lemma}
	
We end the preliminaries by considering equivalent conditions for when a $\Theta$-separator reveals a $3$-connected minor.

\begin{lemma}
\label{theta reveal}
Let $M$ be a $3$-connected matroid such that $r(M),r^*(M)\geq 4$, and let $N$ be a $3$-connected minor of $M$. Let $S\subseteq E(M)$ such that $M|S$ is isomorphic to either $\Theta_n$ or $\Theta^-_n$, where $n\ge 3$. Suppose that $W$ and $Z$ are the sets of segment and cosegment elements of $M|S$, respectively. Then the following statements are equivalent:
\begin{enumerate}[{\rm (i)}]
\item At least one element of $Z$ is $N$-revealing in $M$.

\item The cosimplification $\co(M\delete z)$ has an $N$-minor for at least two elements $z\in Z$.

\item  Both $\si(M/z)$ and $\co(M\delete z)$ have an $N$-minor for all $z\in Z$ and $\co(M\delete w)$ has an $N$-minor for all $w\in W$.
\end{enumerate}
Moreover, if $|E(N)|\leq 3$, then {\rm(i)--(iii)} always hold.
\end{lemma}

\begin{proof}
Since $M|S$ is isomorphic to either $\Theta_n$ or $\Theta_n^-$, there is a labelling of the elements $w_1, w_2, \ldots, w_k$ of $W$ and $z_1, z_2, \ldots, z_n$ of $Z$ such that $(Z-\{z_i\})\cup \{w_i\}$ is a circuit of $M$ for all $i\in \{1, 2, \ldots, k\}$, where $k\in \{n, n-1\}$. Now (iii) certainly implies (ii). Assume that (ii) holds, and let $i\in \{1, 2, \ldots, k\}$. It is straightforward to observe that, as $r(M), r^*(M)\geq 4$, the partition
$$((Z-\{z_i\})\cup \{w_i\}, \{z_i\}, E(M)-(Z\cup \{w_i\}))$$
of $E(M)$ is a cyclic $3$-separation of $M$. Thus, by Lemma~\ref{vertical1}, $\co(M\delete z_i)$ is not $3$-connected, and so (ii) implies (i).

We next show that (i) implies (iii) when $|E(N)|\geq 4$. Let $i\in\{1, 2, \ldots, k\}$, and suppose that $\co(M\delete z_i)$ has an $N$-minor. Since $N$ is simple and $Z-\{z_i\}$ is a series class of $M\delete z_i$, and so $(Z-\{z_i\})\cup \{w_i\}$ is $2$-separating in $M\delete z_i$, it follows that $|(Z\cup\{w_i\})\cap E(N)|\leq 1$. Therefore, by the dual of Lemma~\ref{BS lem 4.5+}, $M\delete w_i$ has an $N$-minor, and both $M\delete z_j$ and $M/z_j$ have an $N$-minor for all $j\in \{1, 2, \ldots, n\}-\{i\}$. In particular, $z_j$ is $N$-revealing for all $j\in \{1, 2, \ldots, k\}$. Thus, the initial choice of $i\in \{1, 2, \ldots, k\}$ was arbitrary. Hence both $M/z$ and $M\delete z$ have an $N$-minor for all $z\in Z$, and $M\delete w$ has an $N$-minor for all $w\in W$. Thus (iii) holds, and so (i)--(iii) are equivalent if $|E(N)|\ge 4$.

We complete the proof by showing, without assuming (i), that if $|E(N)|\le 3$, then (iii) holds and so, by above, (ii) and (i) also hold. To do this, it suffices to show that, for all $z \in Z$ and $w\in W$, each of $\co(M\delete z)$, $\si(M/z)$, and $\co(M\delete w)$ has a $U_{1, 3}$- and $U_{2, 3}$-minor. By Lemma~\ref{exception} and Bixby's Lemma, $\si(M/z)$ and $\co(M\delete w)$ are both $3$-connected. Furthermore, $\si(M/z)$ has rank at least three and $\co(M\delete w)$ has corank at least three. It now follows that $\si(M/z)$ and $\co(M\delete w)$ each have $U_{1, 3}$ and $U_{2, 3}$ as minors. To show that $\co(M\delete z)$ has a $U_{1, 3}$- and a $U_{2, 3}$-minor, it suffices to show that, as $\co(M\delete z)$ is connected, it has rank and corank at least two. Since $r^*(M)\geq 4$, $\co(M\delete z)$ has corank at least three. As $z$ is in at least one circuit of the form $(Z-\{z_i\})\cup \{w_i\}$ in $M$, we have by orthogonality with these circuits that any triad containing $z$ must either be contained in $Z$ or contain an element of $W$. It follows that $z$ is in at most one triad of $M$ with an element outside of $Z$. In particular, $\co(M\delete z)$ has rank at least $r(M)-(n-2)=r(M\delete Z)$ when $n\geq 4$ and rank at least $r(M)-2$ when $n=3$. Thus $\co(M\delete z)$ has rank at least two, completing the proof of the lemma.
\end{proof}

A consequence of the last lemma is that every $\Theta$-separator reveals each of the matroids $U_{0, 0}$, $U_{0, 1}$, $U_{1, 1}$, $U_{1, 2}$, $U_{1, 3}$, and $U_{2, 3}$. We freely use this fact throughout the remainder of the paper.

\section{Proofs of Theorems~\ref{thm: main}, \ref{thm: maximal}, and~\ref{fans1}}
\label{sec: splitter}

In this section we prove Theorems~\ref{thm: main}--\ref{fans1} . We begin with three lemmas. The first two lemmas concern elastic elements in matroids with rank and corank at least four. 

\begin{lemma}
\label{smallN1}
Let $M$ be a $3$-connected matroid such that $r(M), r^*(M)\ge 4$, and let $N$ be a $3$-connected minor of $M$ with at most three elements. Then every elastic element of $M$ is $N$-elastic.
\end{lemma}

\begin{proof}
Let $x$ be an elastic element of $M$. Then $\si(M/x)$ and $\co(M\delete x)$ are both $3$-connected. Furthermore, as $r(M), r^*(M)\ge 4$, we have that $\si(M/x)$ has rank at least three and $\co(M\delete x)$ has corank at least three. Thus, as $\si(M/x)$ is $3$-connected, $\si(M/x)$ contains a circuit, but $\si(M/x)$ is not a circuit, and so $\si(M/x)$ has a $U_{2, 3}$- and a $U_{1, 3}$-minor. By duality, $\co(M\delete x)$ also has a $U_{1, 3}$- and a $U_{2, 3}$-minor. As $|E(N)|\le 3$, it follows that $N$ is a minor of either $U_{1, 3}$ or $U_{2, 3}$, and the lemma follows.
\end{proof}

Theorem~\ref{fans1} is an immediate consequence of Lemma~\ref{elastic fans} and the next lemma.

\begin{lemma}
\label{lem: N-elastic}
Let $M$ be a $3$-connected matroid of corank at least four, and let $N$ be a $3$-connected minor of $M$. Let $(X,\{e\},Y)$ be a vertical $3$-separation of $M$ such that $M/e$ has an $N$-minor and $|X\cap E(N)|\leq 1$. If $Y\cup \{e\}$ is closed, then every elastic element in $X$ is $N$-elastic.
\end{lemma}

\begin{proof}
Let $x$ be an elastic element of $X$. If $|E(N)|\leq 3$, then, by Lemma~\ref{smallN1}, $x$ is $N$-elastic. Thus we may assume that $|E(N)|\ge 4$. In particular, $N$ is simple and cosimple, and so if $M/x$ or $M\delete x$ has an $N$-minor, then $\si(M/x)$ and $\co(M\delete x)$ has an $N$-minor, respectively. Therefore, by Lemma~\ref{BS lem 4.5+}, $x$ is $N$-elastic unless $x$ is the unique exception in the statement of Lemma~\ref{BS lem 4.5+}, in which case, $x\in \cl^*(Y)$ and $e\in \cl(X-\{x\})$. Suppose that $x$ is this unique exception. Then, as $x\in \cl^*(Y)$, it follows by Lemma~\ref{orthogonality}, that $x\not\in \cl(X-\{x\})$. Therefore, as $(Y\cup \{e\}, X)$ is a $3$-separation of $M$, we have
\begin{align*}
2 & = r(Y\cup \{e\})+r(X)-r(M) \\
& = r(Y\cup \{e\})+r(X-\{x\})+1-r(M).
\end{align*}
In particular,
$$1 = r(Y\cup \{e\})+r(X-\{x\})-r(M\delete x),$$
and so $(Y\cup \{e\}, X-\{x\})$ is a $2$-separation of $M\delete x$. Since $e\in \cl(X-\{x\})$, the partition $(Y, (X-\{x\})\cup \{e\})$ is also a $2$-separation of $M\delete x$. Now, as $x$ is elastic, $\co(M\delete x)$ is $3$-connected, and so at least one of $Y\cup \{e, x\}$ and $X$ has corank~$2$, and at least one of $Y\cup \{x\}$ and $X\cup \{e\}$ has corank~$2$. By Lemma~\ref{orthogonality}, $e\not\in \cl^*(X)\cup \cl^*(Y)$. Thus $r^*(X)=r^*(Y\cup \{x\})=2$. But then, as $M$ is $3$-connected, $r^*(M)=3$, contradicting the assumption that $M$ has corank at least four. Hence $x$ is not the exception, and the lemma holds.
\end{proof}

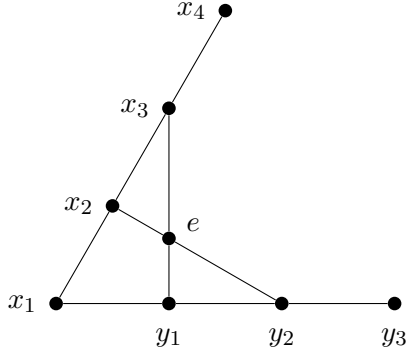
\begin{figure}[ht]
	\begin{tikzpicture}[scale=1.5]
	\begin{scope}[every node/.style=element]
	
	\node (x3) at (60:2) {};
	\node (x2) at (60:1) {};
	\node (x1) at (0,0) {};
	\node (y1) at (1,0) {};
	\node (y2) at (2,0) {};
	\node (y3) at (3,0) {};
	\node (x4) at (60:3) {};
	\node (e) at (intersection of  x2--y2 and y1--x3) {};
	\draw (x4) to (x3) to (x2) to (x1) to (y1) to (y2) to (y3);
	\draw (y2) to (e) to (x2);
	\draw (y1) to (e) to (x3);
	\end{scope}
	\begin{scope}
	\node[scale=1.0] at ($(x1)+(180:0.30)$) {$x_1$};
	\node[scale=1.0] at ($(x2)+(180:0.30)$) {$x_2$};
	\node[scale=1.0] at ($(x3)+(180:0.30)$) {$x_3$};
	\node[scale=1.0] at ($(x4)+(180:0.30)$) {$x_4$};
	\node[scale=1.0] at ($(e)+(30:0.25)$) {$e$};
	\node[scale=1.0] at ($(y1)+(-90:0.3)$) {$y_1$};
	\node[scale=1.0] at ($(y2)+(-90:0.3)$) {$y_2$};
	\node[scale=1.0] at ($(y3)+(-90:0.3)$) {$y_3$};
	\end{scope}
	\end{tikzpicture}
	\caption{The $3$-connected matroid $L_8$.}
	\label{fig: L8}
\end{figure}

The condition in the statement of Lemma~\ref{lem: N-elastic} that $M$ has corank at least four is necessary. To see this, let $L_8$ denote the $3$-connected rank-$3$ matroid shown in Fig.~\ref{fig: L8}. Let $X=\{x_1, x_2, x_3, x_4\}$ and $Y=\{y_1, y_2, y_3\}$. Then $(X, \{e\}, Y)$ is a cyclic $3$-separation of $L_8$, and $L_8\delete e$ has a $U_{2, 4}$-minor whose ground set contains $Y$. The element $x_1$ of $L_8$ is elastic but it is not $U_{2, 4}$-elastic. However, every element of $X-\{x_1\}$ is $U_{2,4}$-elastic. The next lemma captures this last observation and is the corank-three analogue of Lemma~\ref{lem: N-elastic}.

\begin{lemma}
\label{lem: corank3}
Let $M$ be a $3$-connected rank-$3$ matroid, and let $N$ be a $3$-connected minor of $M$. Let $(X,\{e\},Y)$ be a cyclic $3$-separation of $M$ such that $M\delete e$ has an $N$-minor and $|X\cap E(N)|\leq 1$. If $X\cup \{e\}$ is not a $4$-element fan, then there is at most one element of $X$ that is not $N$-elastic. Moreover, if such an element $x$ exists, then $x\in \cl(Y)$.
\end{lemma}

\begin{proof}
Since $M$ is $3$-connected and $(X, \{e\}, Y)$ is a cyclic $3$-separation of $M$, we have $r(X)=r(Y)=2$, and $|X|, |Y|\ge 3$. Furthermore, $|\cl(Y)\cap X|\le 1$. Now, as $M\delete e$ has an $N$-minor and $|X\cap E(N)|\leq 1$, it follows that $N$ is a minor of $U_{2, n}$, where $n=|\cl(Y)|\le |Y|+1$. Let $x\in X-\cl(Y)$. As $X\cup\{e\}$ is not a $4$-element fan, $|X-\cl(Y)|\ge 3$, and so $\co(M\delete x)=M\delete x$ and $M\delete x$ is $3$-connected. Furthermore, $\si(M/x)$ is isomorphic to either $U_{2, n}$ or $U_{2, n+1}$ depending on whether or not $e$ is in a triangle with $x$. In particular, $\si(M/x)$ and $\co(M\delete x)$ are both $3$-connected with $N$-minors. This completes the proof of the lemma.
\end{proof}

We next prove Theorem~\ref{thm: maximal}.

\begin{proof}[Proof of Theorem~\ref{thm: maximal}]
Let $(X, \{e\}, Y)$ be a vertical $3$-separation of $M$ such that $M/e$ has an $N$-minor and $|X\cap E(N)|\le 1$. Without loss of generality, we may assume that $Y\cup\{e\}$ is closed. Now let $(X', \{e'\}, Y')$ be a vertical $3$-separation of $M$ such that $Y\cup \{e\}\subseteq Y'\cup \{e'\}$ and $Y'\cup \{e'\}$ is maximal, and suppose that $X'\cup \{e'\}$ is not a $4$-element fan. If $r^*(M)=3$, then, by Lemma~\ref{lem: corank3}, $X'$ contains at least two $N$-elastic elements. Thus, we may assume that $r^*(M)\ge 4$. Assume that $X'$ is contained in a $\Theta$-separator $S$. If $|E(N)|\leq 3$, then $S$ reveals $N$. Say $|E(N)|\geq 4$, in which case, $N$ is simple and cosimple. By Lemma~\ref{lem: thetamax}, either
\begin{enumerate}[(I)]
\item $M^*|S$ is isomorphic to either $\Theta_n$ or $\Theta^-_n$, where $n=|X'\cup \{e'\}|-1$, and there is a unique element $x\in X'$ such that $x$ is a segment element of $M^*|S$ and $(X'-\{x\})\cup \{e'\}$ is the set of cosegment elements of $M^*|S$, or

\item $X'\cup \{e'\}$ is a circuit, $M|S$ is isomorphic to either $\Theta_n$ or $\Theta^-_n$ for some $n\in \{|X'|, |X'|+1\}$, and $X'$ is a subset of the cosegment elements of $M|S$.
\end{enumerate}
If (I) holds, then, for all $x'\in X'-\{x\}$, it follows by Lemma~\ref{BS lem 4.5+} that $M/x'$, and hence $\si(M/x')$, has an $N$-minor. Thus Lemma~\ref{theta reveal}(ii) holds, and so $S$ reveals $N$. If (II) holds, then, by Lemma~\ref{BS lem 4.5+}, there are at least two elements $x'\in X'$ such that $M\delete x'$, and hence $\co(M\delete x')$, has an $N$-minor. Again, we deduce by Lemma~\ref{theta reveal} that $S$ reveals $N$. Thus we may assume that $X'$ is not contained in any $\Theta$-separator. Then, as $Y'\cup\{e'\}$ is maximal, it follows by Lemma~\ref{elastic1} that $X'$ contains at least two elastic elements. By Lemma~\ref{lem: N-elastic}, each of these elastic elements is $N$-elastic, thereby completing the proof of the theorem.\end{proof}

Lastly, we use Theorem~\ref{thm: maximal} to prove Theorem~\ref{thm: main}.

\begin{proof}[Proof of Theorem~\ref{thm: main}]
Let $e$ be an $N$-revealing element of $M$. Then, up to duality, $\si(M/e)$ has an $N$-minor and is not $3$-connected. It follows by Lemmas~\ref{vertical1} and \ref{2sep} that $M$ has a vertical $3$-separation $(X, \{e\}, Y)$ such that $|X\cap E(N)|\leq 1$. Choosing $(X-\cl(Y), \{e\}, \cl(Y)-\{e\})$ if necessary, $M$ has a vertical $3$-separation $(X', \{e'\}, Y')$ such that $Y\cup \{e\}\subseteq Y'\cup \{e'\}$ and $Y'\cup \{e'\}$ is maximal. Since $M$ has no $4$-element fans or $\Theta$-separators revealing $N$, it follows by Theorem~\ref{thm: maximal} that $X'$ contains at least two $N$-elastic elements, completing the proof of the theorem.
\end{proof}

\section{Matroids with the Smallest Number of Elastic Elements}
\label{sec: min elts}

In this section, we prove Theorems~\ref{min eltsI} and~\ref{min eltsII}. We begin with three lemmas. The first lemma follows easily from the definitions and the second is \cite[Lemma 6.3]{bre14}. A proof of the third lemma can be found in~\cite{whi99}.

\begin{lemma}
\label{BS lem 6.1}
Let $(X, Y)$ be a partition of the ground set of a matroid $M$ such that $|X|, |Y|\geq 2$. Then $(X, Y)$ is a sequential $3$-separation of $M$ if and only if for some $U\in \{X, Y\}$, there is a path of $3$-separations $(P_0, P_1, \ldots, P_k, U)$ in $M$ such that $|P_0|=2$ and $|P_i|=1$ for all $i\in\{1, 2, \ldots, k\}$.
\end{lemma}

\begin{lemma}
\label{BS lem 6.3}
Let $M$ be a $3$-connected matroid with distinct elements $s_1$ and $s_2$. Let $U$ be a subset of $E(M)-\{s_1, s_2\}$ such that $|E(M)-(U\cup\{s_1, s_2\})|\geq 2$. If, for each $u\in U$, there is a path of $3$-separations $(X_u, \{u\}, Y_u)$ in $M$ such that $\{s_1, s_2\}\subseteq X_u\subseteq U\cup \{s_1, s_2\}$, then there is an ordering $(u_1, u_2, \ldots, u_k)$ of $U$ such that
$$(\{s_1, s_2\}, \{u_1\}, \{u_2\}, \ldots, \{u_k\}, E(M)-(U\cup\{s_1, s_2\}))$$
is a path of $3$-separations in $M$.
\end{lemma}

\begin{lemma}
\label{triangle}
Let $C^*$ be a rank-$3$ cocircuit of a $3$-connected matroid $M$. If $e\in C^*$ has the property that $\cl_M(C^*)-\{e\}$ contains a triangle of $M/e$, then $\si(M/e)$ is $3$-connected
\end{lemma}

We now prove Theorem~\ref{min eltsI}.
\begin{proof}[Proof of Theorem~\ref{min eltsI}]
Suppose that $M$ has exactly four elastic elements $f_1$, $f_2$, $g_1$, and $g_2$. Then, for all $e\in E(M)-\{f_1, f_2, g_1, g_2\}$ either $\si(M/e)$ is not $3$-connected or $\co(M\delete e)$ is not $3$-connected. Thus, by Lemmas~\ref{BS lem 6.1}, \ref{BS lem 6.3} and \ref{vertical1} it suffices to show that there is a partition, say $(\{f_1,f_2\},\{g_1,g_2\})$, of the set of elastic elements of $M$ such that every vertical $3$-separation or cyclic $3$-separation of $M$ is of the form $(X,\{e\},Y)$, where $\{f_1, f_2\}\subsetneqq X$ and $\{g_1, g_2\}\subsetneqq Y$. Suppose that this fails. Then there exists two partitions $(X_1, \{e_1\}, Y_2)$ and $(X_2, \{e_2\}, Y_2)$ of $E(M)$ such that each partition is either a vertical or a cyclic $3$-separation of $M$, and each of the intersections $X_1\cap X_2$, $X_1\cap Y_2$, $Y_1\cap X_2$, and $Y_1\cap Y_2$ contains a unique elastic element. Without loss of generality, we may assume that $e_1\in Y_2$, $e_2\in X_1$, $f_1\in X_1\cap X_2$, $f_2\in X_1\cap Y_2$, $g_1\in Y_1\cap X_2$ and $g_2\in Y_1\cap Y_2$, and that, up to duality, $(X_1, \{e_1\}, Y_1)$ is a vertical $3$-separation.

By uncrossing $X_2\cup\{e_2\}$ with each of $X_1$ and $X_1\cup\{e_1\}$, we have that $(Y_1\cap Y_2)\cup\{e_1\}$ and $Y_1\cap Y_2$ are $3$-separating. If $r(Y_1\cap Y_2)\geq 3$, then $(Y_1\cap Y_2, \{e_1\}, X_1\cup X_2)$ is a vertical $3$-separation of $M$ and so, by Lemma~\ref{elastic1}, $Y_1\cap Y_2$ contains at least two elastic elements, a contradiction. We deduce that $r((Y_1\cap Y_2)\cup\{e_1\})=2$. Furthermore, if $Y_1\cap X_2=\{g_1\}$, then, as $M$ has no $4$-element fans, $|Y_1\cap Y_2|\ge 3$ and $Y_1\cap Y_2$ contains a cycle. Also, as $e_1\in \cl(X_1)$, we have that $X_1\cup \{e_1\}$ contains a cycle, and so $(Y_1-\{g_1\}, \{g_1\}, X_1\cup\{e_1\})$ is a cyclic $3$-separation of $M$, contradicting the fact that $g_1$ is elastic. Thus $|Y_1\cap X_2|\geq 2$.

Next, by uncrossing $Y_1$ with each of $X_2$ and $X_2\cup\{e_2\}$, we see that $(X_1\cap Y_2)\cup\{e_1, e_2\}$ and $(X_1\cap Y_2)\cup\{e_2\}$ are exactly $3$-separating. If $r((X_1\cap Y_2)\cup \{e_2\})\geq 3$, then $((X_1\cap Y_2)\cup \{e_2\}, \{e_1\}, X_2\cup Y_1)$ is a vertical $3$-separation of $M$ and it follows by Lemma~\ref{elastic1} that $X_1\cap Y_2$ contains at least two elastic elements, a contradiction. We deduce that $r((X_1\cap Y_2)\cup \{e_1, e_2\})=2$. In particular, $e_2\in \cl(\{e_1, f_2\})$, and so $(X_2, \{e_2\}, Y_2)$ is a vertical $3$-separation of $M$. By Lemma~\ref{vertical2}, we may assume that $X_2\cup \{e_2\}$ is closed.

If $Y_1\cap Y_2=\{g_2\}$, then either $Y_2\cup\{e_2\}$ is a $4$-element fan, a contradiction, or $(Y_2-\{g_2\}, \{g_2\}, X_2\cup \{e_2\})$ is a cyclic $3$-separation of $M$, contradicting the fact that $g_2$ is elastic. Hence $|Y_1\cap Y_2|\geq 2$. Lastly, if $X_1\cap Y_2=\{f_2\}$, then $(Y_2-\{f_2\}, \{f_2\}, X_2\cup \{e_2\})$ is a cyclic $3$-separation of $M$, another contradiction. Thus, $|X_1\cap Y_2|\geq 2$ and the set $(X_1\cap Y_2)\cup\{e_1,e_2\}$ is a segment of at least four elements. As $X_2\cup \{e_2\}$ is closed, it follows that $Y_2$ is a rank-$3$ cocircuit and so, as $|Y_1\cap Y_2|\ge 3$, Lemmas~\ref{segdelete} and~\ref{triangle} imply that every element of $X_1\cap Y_2$ is elastic. This final contradiction completes the proof. 
\end{proof}

The next lemma is used in the proof of Theorem~\ref{robust}. It is the analogue of Theorem~\ref{min eltsII} for when the size of $E(N)$ is at most three.

\begin{proposition}
\label{smallN2}
Let $M$ be a $3$-connected matroid with no $4$-element fans and no $\Theta$-separators, and let $N\in\{U_{0, 1}, U_{1, 1}, U_{1, 2}, U_{1, 3}, U_{2, 3}\}$. If $r(M), r^*(M)\geq 3$ and $|E(M)|\geq 8$, then $M$ has at least four $N$-elastic elements. Moreover, if $M$ has exactly four $N$-elastic elements, then $M$ has path-width three.
\end{proposition}

\begin{proof}
By Theorem~\ref{elastic2}, $M$ has at least four elastic elements. If every elastic element of $M$ is $N$-elastic, then the proposition follows from Theorem~\ref{min eltsI}. Thus we may assume that $M$ has an elastic element $e$ which is not $N$-elastic. By Lemma~\ref{smallN1}, we may assume that, up to duality, $r(M)=3$. Since $|E(M)|\geq 8$, it follows that $r^*(M)\geq 5$, and so the $3$-connected matroid $\co(M\delete e)$ has corank at least four and rank at least two. In particular, $\co(M\delete e)$ has a $U_{1, 3}$- and a $U_{2, 3}$-minor, and therefore an $N$-minor. Hence, as $e$ is not $N$-elastic, we have that $\si(M/e)$ has no $N$-minor. This is only possible if $\si(M/e)\cong U_{2,3}$ and $N\cong U_{1,3}$, in which case $M$ is comprised of a triangle, say $\{e_1, e_2, e_3\}$, and three segments $L_1$, $L_2$, and $L_3$ such that $L_1\cap L_2\cap L_3=\{e\}$ and $e_i\in L_i$ for all $i\in \{1, 2, 3\}$. As $M$ has at least eight elements, at least one of these segments, say $L_1$, has at least four elements. A routine check shows that every element of $E(M)-L_1$ as well as at least one element of $L_1$ is $U_{2, 4}$-elastic, and thus $U_{1, 3}$-elastic and $U_{2, 3}$-elastic. Hence $M$ has at least five $N$-elastic elements, completing the proof of the lemma.
\end{proof}

To see that the requirement of Proposition~\ref{smallN2} that $M$ have rank and corank at least three is necessary, consider the case when $M$ is isomorphic to $U_{2,  5}$ and $N$ is isomorphic to $U_{1, 3}$. If $e\in E(M)$, then $M/e$, which is isomorphic to $U_{1, 4}$, has a $U_{1, 3}$-minor but $\si(M/e)$, which is isomorphic to $U_{1, 1}$, has no $U_{1, 3}$-minor. Thus, in this case, $M$ has no $N$-elastic elements. To see that the requirement $|E(M)|\geq 8$ is necessary, consider the case when $M$ is isomorphic to $F_7$ and $N$ is isomorphic to $U_{1, 3}$. If $e\in E(M)$, then $M/e$ has a $U_{1, 3}$-minor but $\si(M/e)$, which is isomorphic to $U_{2, 3}$, has no $U_{1, 3}$-minor. Thus, again, $M$ has no $N$-elastic elements.

We next prove Theorem~\ref{min eltsII}.

\begin{proof}[Proof of Theorem~\ref{min eltsII}]
Suppose that $M$ has exactly two $N$-elastic elements $s_1$ and $s_2$. We first show that $|E(M)-(K\cup \{s_1, s_2\})|\geq 2$. If $M$ has no $\Theta$-separators, then, by Theorem~\ref{elastic2}, $M$ has at least four elastic elements and so $|E(M)-(K\cup \{s_1, s_2\})|\ge 2$. Therefore assume that $M$ has a $\Theta$-separator. Let $W$ be a rank-$2$ subset and $Z$ a corank-$2$ subset of $E(M)$ such that $W\cup Z$ is a $\Theta$-separator of $M$. Since $M$ has no $4$-element fans, it follows that $\min\{|W|, |Z|\}\geq 3$. By Lemma~\ref{theta elastic}, at most one element of $W\cup Z$ is elastic. Therefore, if $|E(M)-(K\cup \{s_1, s_2\})|\leq 1$, then at least one element of $W$ and at least one element of $Z$ is $N$-revealing. As $M$ has no $\Theta$-separators revealing $N$, we deduce that $|E(M)-(K\cup \{s_1, s_2\})|\geq 2$. 
	
Next, for each $e\in K$, we select a certain path of $3$-separations $(X_e, \{e\}, Y_e)$. Let $e\in K$. Up to duality, we may assume that $\si(M/e)$ has an $N$-minor and is not $3$-connected. Then, by Lemmas~\ref{vertical1}, \ref{vertical2} and~\ref{2sep}, there is a vertical $3$-separation $(X, \{e\}, Y)$ of $M$ such that $Y\cup \{e\}$ is closed and $|X\cap E(N)|\leq 1$. By Theorem~\ref{thm: maximal}, $X$ contains at least two $N$-elastic elements. Thus $\{s_1, s_2\}\subseteq X$. Furthermore, by Lemma~\ref{BS lem 4.5+}, $M/x$ has an $N$-minor for all $x\in X$ and there is at most one element, $x'$ say, for which $M\delete x'$ has no $N$-minor. If there is no such element $x'$, then let $X_e=X$ and $Y_e=Y$. Otherwise, we note by Lemma~\ref{BS lem 4.5+}, that $x'\in \cl^*(Y)$ and thus $(X-\{x'\},\{e\},Y\cup\{x'\})$ is a path of $3$-separations by Lemma~\ref{3sep1}. In this case, let $X_e=X-\{x'\}$ and $Y_e=Y\cup \{x'\}$. Observe that, by this selection process, $\{s_1,s_2\}\in X_e$ and, for all $x\in X_e$, both $M/x$ and $M\delete x$ have an $N$-minor. Moreover, as $|E(N)|\geq 4$, the latter property implies that $X_e \subseteq K\cup\{s_1,s_2\}$.
	
Now let $Y=E(M)-(K\cup\{s_1, s_2\})$. By an application of Lemma~\ref{BS lem 6.3}, there is an ordering $(e_1, e_2, \ldots, e_k)$ of $K$ such that $(\{s_1, s_2\}, \{e_1\}, \{e_2\}, \ldots, \{e_k\}, Y)$ is a path of $3$-separations in $M$. It remains to show that $M\delete e_i$ and $M/e_i$ have an $N$-minor for all $i < k$. Dualising if necessary, we may assume that $e_k\in \cl(Y)$. We may also assume that $k\geq 2$.

\begin{sublemma}
Let $A=\{s_1, s_2, e_1, e_2, \ldots, e_k\}$ and let $L=\cl(Y)-Y$. If $|L|\ge 2$, then $M\delete \ell$ and $M/\ell$ have an $N$-minor for all $\ell\in L$. 
\label{sub}
\end{sublemma}

Since $e_k\in L$, it follows by submodularity that
\begin{align*}
1\le r(L)=r(\cl(Y)\cap A)\leq r(Y) + r(A) - r(M)=2.
\end{align*}
Let $\ell\in L$, and suppose that $|L|\ge 2$. Now either $(A-\{\ell\}, \{\ell\}, Y)$ is a vertical $3$-separation, or at least one of $A$ and $Y$ has rank two. Thus, by a combination of Lemma~\ref{vertical1} and Bixby's Lemma or by Lemma~\ref{segdelete}, the matroid $\co(M\delete \ell)$ is $3$-connected. Therefore, by the definition of $K$, the matroid $\si(M/\ell)$, and hence $M/\ell$, has an $N$-minor. If either $A$ or $Y$ has rank two, then, as $N$ is simple and $M/\ell$ has an $N$-minor, $M\delete \ell'$ has an $N$-minor for all $\ell'\in L-\{\ell\}$. Furthermore, if $(A-\{\ell\}, \{\ell\}, Y)$ is a vertical $3$-separation, then $(A-\{\ell\}, Y)$ is a $2$-separation of $M/\ell$ in which $\ell'\in \cl(A-\{\ell, \ell'\})\cap \cl(Y)$ for all $\ell'\in L-\{\ell\}$. Since $N$ is $3$-connected, either $|(A-\{\ell\})\cap E(N)|\le 1$ or $|Y\cap E(N)|\le 1$. Also, as $M/\ell/\ell'$ is not connected, $M/\ell\delete \ell'$ is connected. Thus, by Lemma~\ref{2sep}, $M\delete \ell'$ has an $N$-minor for all $\ell'\in L-\{\ell\}$. Since the choice of $\ell$ was arbitrary, we deduce that if $|L|\ge 2$, then $M\delete \ell$ and $M/\ell$ have an $N$-minor for all $\ell\in L$.

If $Y$ spans $E(M)$, then, as $k\ge 2$, it follows by (\ref{sub}) that $M\delete e_i$ and $M/e_i$ have an $N$-minor for all $i < k$. Thus we may suppose that $Y$ does not span $E(M)$. Let $j$ be the highest index such that $e_j\not\in \cl(Y)$. Let $A_j=\{s_1, s_2, e_1, e_2, \ldots, e_j\}$ and let $B_j=\{e_{j+1}, e_{j+2}, \ldots, e_k\}\cup Y$. If $A_j$ is independent, then, as $r(A_j)+r(B_j)-r(M)=2$, it follows that $r^*(A_j)=2$. In this case, as $M\delete s_1$ has an $N$-minor, $M/e_i$ has an $N$-minor for all $i\le j$. Furthermore, if $|A_j|\ge 2$, then, by the dual of Lemma~\ref{segdelete}, $M/e_i$ is $3$-connected for all $i\le j$. So, by the definition of $K$, the matroid $M\delete e_i$ has an $N$-minor for all $i\le j$. If $|A_j|=1$, then $j=1$, $\{s_1, s_2, e_1\}$ is a triad of $M$, and $e_k\in \cl(\{s_1, s_2, e_1\})$. Since $M$ has no $4$-element fans, it follows by Lemma~\ref{triangle} that $M/e_1$ is $3$-connected. Thus, again by the definition of $K$, the matroid $M\delete e_1$ has an $N$-minor. Therefore, we may assume that $A_j$ is dependent. If $A_j-\{e_j\}$ is independent, then $e_j\in \cl(A_j-\{a_j\})$ and $e_j\in \cl^*(B_j)$, contradicting orthogonality. Therefore $A_j-\{a_j\}$ is dependent and so, as $B_j$ is dependent, $\{A_j-\{e_j\}, \{e_j\}, B_j\}$ is a cyclic $3$-separation of $M$. Since $\co(M\delete e_j)$ is not $3$-connected, it follows by Bixby's Lemma and the definition of $K$ that $M\delete e_j$ has an $N$-minor. If $|B_j\cap E(N)|\le 1$, then, by the dual of Theorem~\ref{thm: maximal}, $B_j$ contains two $N$-elastic elements. But $M$ has exactly two $N$-elastic elements $s_1$ and $s_2$, and $\{s_1, s_2\}\subseteq A_j$. Hence $|(A_j-\{e_j\})\cap E(N)|\le 1$.

Now let $L'=\cl^*(B_j)-B_j$. If $|L'|\ge 2$, then, by a similar, but dual, argument used to established~(\ref{sub}), $M\delete \ell'$ and $M/\ell'$ have an $N$-minor for all $\ell'\in L'$. Furthermore, if $|L'|\ge 2$, then $|\cl(A_j-L')\cap \cl(B_j)|=0$, while if $|L'|=1$, then $|\cl(A_j-\{e_j\})\cap \cl(B_j)|\le 1$. If, for some $i\in \{1, 2, \ldots, k-1\}$, we have $\{e_i\}=\cl(A_j-\{e_j\})\cap \cl(B_j)$, then, as $e_i\in \cl(Y)$, it follows by~(\ref{sub}) that $M\delete e_i$ and $M/e_i$ have an $N$-minor. Thus, by the dual of Lemma~\ref{BS lem 4.5+}, to complete the proof of the theorem it suffices to show that $M/e_j$ has an $N$-minor when $|L'|=1$. Since $\si(M/e_j)$ is $3$-connected, $M/e_j/e_k$ is connected. Therefore, as $(A_j, Bj-\{e_k\})$ is a $2$-separation of $M/e_k$, Lemma~\ref{2sep} implies that $M/e_j$ has an $N$-minor.
\end{proof}

\section{Applications to Fixed-Basis Theorems}
\label{sec: fixed basis}

In this section, we show that two previous results concerning the removal of elements relative to a fixed basis are consequences of the results in this paper and~\cite{dru21}. Moreover, we also resolve a question posed in~\cite{whi13}. Let $M$ be a $3$-connected matroid and let $B$ be a basis of $M$. Following \cite{whi13}, we say that an element $e$ of $M$ is \emph{removable with respect to $B$} if either 
\begin{enumerate}[{\rm (i)}]
\item $e\in B$ and $\si(M/e)$ is $3$-connected, or

\item $e\in E(M)-B$ and $\co(M\delete e)$ is $3$-connected.
\end{enumerate}
Of course, if an element is elastic, then it is removable with respect to any basis. However, using an appropriate choice for a basis $B$, it is easily checked that a $5$-element fan may have no removable elements with respect to $B$. Nevertheless, removable elements are abundant in all larger $\Theta$-separators. The straightforward proof of the following lemma is omitted.

\begin{lemma}
\label{theta A}
Let $M$ be a $3$-connected matroid and let $B$ be a basis of $M$. Let $W$ be a rank-$2$ subset and let $Z$ be a corank-$2$ subset of $M$ such that $W\cup Z$ is a $\Theta$-separator of $M$ with at least six elements. Then
\begin{enumerate}[{\rm (i)}]
\item $|\cl(W)-B|\geq |\cl(W)|-2$ and $\co(M\delete w)$ is $3$-connected for all $w\in \cl(W)$, and

\item $|B \cap \cl^*(Z)|\geq |\cl^*(Z)|-2$ and $\si(M/z)$ is $3$-connected for all $z\in \cl^*(Z)$.
\end{enumerate}
\end{lemma}

We now show that the main result of \cite{whi13} follows from Theorems~\ref{elastic2} and~\ref{min eltsI}, and a treatment of $\Theta$-separators.

\begin{theorem}[\cite{whi13}, Theorem 1.1]
Let $M$ be a $3$-connected matroid with no $4$-element fans, where $|E(M)|\geq 4$. Let $B$ be a basis of $M$. Then $M$ has at least four elements that are removable with respect to $B$. Moreover, if $M$ has exactly four removable elements with respect to $B$, then $M$ has path-width three.
\end{theorem}

\begin{proof}
First 	suppose that there is a rank-$2$ subset $W$ and a corank-$2$ subset $Z$ of $E(M)$ such that $W\cup Z$ is a $\Theta$-separator of $M$, in which case, $r(M), r^*(M)\geq 4$. Up to duality, we may assume that $M|(W\cup Z)\in \{\Theta_n,\Theta_n^-\}$. As $M$ has no $4$-element fans, $n\ge 4$. By Lemma~\ref{theta A}, at least $|Z|-2$ elements of $Z$ and at least $|W|-2$ elements of $W$ are removable with respect to $B$. If $Z$ spans $M$, then $|\cl(W)|\geq 4$ and, as $|B\cap \cl(W)|\leq 2$, it follows by Lemma~\ref{segdelete} that $M$ has at least four elements that are removable with respect to $B$. Moreover, in this instance, $M$ has path-width three as a sequential ordering of $E(M)$ is obtained by first progressing through the elements in $Z$, and then through the elements in $\cl(W)$. Thus we may assume that $Z$ does not span $M$. Then, for any $w\in W$, the partition
$$((W\cup Z)-\{w\}, \{w\}, E(M)-(W\cup Z))$$
is a vertical $3$-separation of $M$. Let $(U,\{e\},V)$ be a vertical $3$-separation of $M$ such that $V\cup\{e\}$ is maximal and contains $W\cup Z$. Then, by Lemma~\ref{elastic1}, $U$ has at least two elastic elements, or it is contained in a $\Theta$-separator. In the latter case, by combining Lemmas~\ref{lem: thetamax} and~\ref{theta A}, we see that the set $E(M)-(W\cup Z)$ has at least two elements that are removable with respect to $B$ and so, in both cases, $M$ has at least five elements that are removable with respect to $B$.

To complete the proof, we may now assume that $M$ has no $\Theta$-separators. Then, by Theorem~\ref{elastic2}, $M$ has at least four elastic elements, and so $M$ has at least four elements that are removable with respect to $B$. Moreover, if $M$ has exactly four removable elements with respect to $B$, then these are precisely the elastic elements of $M$, and thus $M$ has path-width~$3$ by Theorem~\ref{min eltsI}.
\end{proof}

In~\cite{whi13}, Whittle and Williams asked if there exists a $3$-connected matroid $M$ with no $4$-element fans such that for every basis $B$ of $M$ there are exactly four elements of $M$ which are removable with respect to $B$. In particular, the next proposition shows that no matroid with at least four elements has this property.

\begin{proposition}
Let $M$ be a $3$-connected matroid with no $4$-element fans such that $|E(M)|\geq 4$. Then there exists a basis $B$ of $M$ such that $M$ has at least five removable elements with respect to $B$. 
\end{proposition}

\begin{proof}
First suppose that $M$ has a rank-$2$ subset $W$ and a corank-$2$ subset $Z$ such that $W\cup Z$ is a $\Theta$-separator of $M$. By duality, we may assume that $M|(W\cup Z)\in\{\Theta_n,\Theta_n^-\}$. Since $M$ has no $4$-element fans, $n\geq 4$. Let $B$ be a basis of $M$ containing the independent set $Z$. Then, by Lemma~\ref{theta A}, every element of $Z$ and at least one element of $W$ is removable with respect to $B$, giving a total of at least five such elements. Thus, we may assume $M$ has no $\Theta$-separators. Then, by Theorem~\ref{elastic2}, $M$ has at least four elastic elements. Since elastic elements are removable with respect to any basis, the proposition holds if $M$ has at least five elastic elements. So assume that $M$ has exactly four elastic elements. The only $3$-connected matroid on four elements is $U_{2, 4}$, which has a $4$-element fan. Thus $M$ has at least five elements. Let $e$ be a non-elastic element of $M$. As $M$ has no coloops, $M$ has a basis containing $e$ as well as a basis avoiding $e$. If $e$ is not removable with respect to any basis, it follows that the matroids $\si(M/e)$ and $\co(M\delete e)$ are not $3$-connected, a contradiction to Bixby's Lemma. Hence $e$ is removable with respect to some basis $B$, in which case, $M$ has at least five removable elements with respect to $B$.
\end{proof}

Let $M$ be a $3$-connected matroid, let $N$ be a $3$-connected minor of $M$, and let $B$ be a basis of $M$. Following \cite{bre14}, an element $e$ of $M$ is called \emph{$(N,B)$-robust} if either
\begin{enumerate}[{\rm (i)}]
\item $e\in B$ and $M/e$ has an $N$-minor, or

\item $e\in E(M)-B$ and $M\delete e$ has an $N$-minor.
\end{enumerate}
Furthermore, such an element is called \emph{$(N,B)$-strong} if either
\begin{enumerate}[{\rm (i)}]
\item $e\in B$ and $\si(M/e)$ is $3$-connected with an $N$-minor, or

\item $e\in E(M)-B$ and $\co(M\delete e)$ is $3$-connected with an $N$-minor.
\end{enumerate}

Evidently, an $N$-elastic element of $M$ is $(N,B)$-strong for every basis $B$ of $M$. The next lemma follows by combining Lemma~\ref{theta A} with Lemma~\ref{theta reveal}.

\begin{lemma}
\label{theta B}
Let $M$ be a $3$-connected matroid, let $N$ be a $3$-connected minor of $M$, and let $B$ be a basis of $M$. Let $S$ be a $\Theta$-separator of $M$ with at least six elements. If $S$ reveals $N$ in $M$, then at least $|S|-4$ elements of $S$ are $(N,B)$-strong.
\end{lemma}

We end the paper by showing that the two main results of~\cite{bre14} follow from Theorems~\ref{thm: main} and~\ref{min eltsII}, and a treatment of $\Theta$-separators.

\begin{theorem}[\cite{bre14}, Theorems 1.1 and 1.2]
Let $M$ be a $3$-connected matroid with no $4$-element fans such that $|E(M)|\geq 5$. Let $N$ be a $3$-connected minor of $M$, and let $B$ be a basis of $M$.
\begin{enumerate}[{\rm (i)}]
\item If $M$ has two distinct $(N, B)$-robust elements, then $M$ has two distinct $(N,B)$-strong elements.

\item Let $P$ denote the set of $(N, B)$-robust elements of $M$. If $M$ has precisely two $(N, B)$-strong elements, then $(P, E(M)-P)$ is a sequential $3$-separation of $M$.
\end{enumerate}
\label{robust}
\end{theorem}

\begin{proof}
If $M$ has rank or corank at most two, then the theorem follows easily from the fact that $|E(M)|\geq 5$. Furthermore, a routine check shows that the theorem holds if $|E(M)|\in \{6, 7\}$. Thus we may assume that $r(M), r^*(M)\geq 3$ and $|E(M)|\geq 8$. Since $M$ has no $4$-element fans, any $\Theta$-separator of $M$ has at least seven elements. Therefore, if $M$ has a $\Theta$-separator revealing $N$, then, by Lemma~\ref{theta B}, $M$ has at least three $(N, B)$-strong elements. Thus we may assume that $M$ has no such $\Theta$-separators. We first prove (i). If $|E(N)|\leq 3$, then, by Proposition~\ref{smallN2}, $M$ has at least four $N$-elastic elements, and so (i) holds. Thus we may assume that $|E(N)|\geq 4$, in which case, every $(N, B)$-robust element is either $(N, B)$-strong or $N$-revealing. It follows that either each of the two guaranteed $(N,B)$-robust elements are $(N,B)$-strong or, by Theorem~\ref{thm: main}, $M$ has at least two $N$-elastic elements. In either instance, $M$ has at least two $(N,B)$-strong elements, thereby proving (i).

To prove (ii), suppose that $M$ has precisely two $(N, B)$-strong elements $\{s_1, s_2\}$, in which case, $|E(N)|\ge 4$. If $M$ has no $N$-revealing elements, then $P=\{s_1, s_2\}$ and $(P, E(M)-P)$ is trivially a sequential $3$-separation of $M$. So assume that $M$ has at least one $N$-revealing element. In this case, it follows by Theorem~\ref{thm: main} that $s_1$ and $s_2$ are $N$-elastic, and that $M$ has no further $N$-elastic elements. Now let $K$ be the set of $N$-revealing elements of $M$. Note that $P-\{s_1,s_2\}\subseteq K$. By Theorem~\ref{min eltsII}, $K$ has an ordering $(e_1, e_2, \ldots, e_k)$ such that
$$(\{s_1, s_2\}, \{e_1\}, \{e_2\}, \ldots, \{e_k\}, E(M)-(K\cup\{s_1, s_2\}))$$
is a path of $3$-separations in $M$ and, for all $i< k$, both $M/e_i$ and $M\delete e_i$ have an $N$-minor. In particular, $e_i$ is $(N,B)$-robust for all $i<k$ and, consequently, $P$ is either $K\cup\{s_1, s_2\}$ or $K\cup\{s_1, s_2\}-\{e_k\}$. Thus, by Lemma~\ref{BS lem 6.1}, $(P,E(M)-P)$ is a sequential $3$-separation, completing the proof of (ii) and the theorem.
\end{proof}

\end{document}